\documentclass[english,a4paper,10pt]{amsart}
\usepackage[latin1]{inputenc}
\usepackage{amssymb,amsthm}
\usepackage[all]{xy}
\usepackage{amsmath}
\usepackage{indentfirst}
\usepackage{amsfonts}
\usepackage{textcomp}
\usepackage{graphicx}
\usepackage{enumerate}
\usepackage{hyperref}

\newtheorem{lemma}{Lemma}[section]

\newtheorem{proposition}[lemma]{Proposition}

\theoremstyle{definition}
\newtheorem{definition}[lemma]{Definition}
\theoremstyle{remark}
\newtheorem{remark}[lemma]{Remark}

\def\mod{\operatorname{mod}}

\begin{document}

\title{Euler pseudoprimes for half of the bases}
\thanks{{\it 2010 Math classification:} 11A15, 11A51, 11Y11}
\thanks{\today}

\author{Lorenzo Di Biagio}
\address{Dipartimento di Matematica, Universit\`a degli Studi ``Roma Tre''- Largo S.\ Leonardo Murialdo 1, 00146, Roma, Italy}
\email[1]{dibiagio@mat.uniroma3.it}
\email[2]{lorenzo.dibiagio@gmail.com}

\begin{abstract}
We prove that an odd number $n$ is an Euler pseudoprime for exactly one half of the admissible bases if and only if $n$ is a special Carmichael number, that is, $a^{\frac{n-1}{2}} \equiv 1 \mod n$ for every invertible $a \in \mathbb{Z}_n$.
\end{abstract}
\dedicatory{Dedicated to the memory of Prof.~ John Lewis Selfridge}

\maketitle


\setcounter{section}{0}
\setcounter{lemma}{0} 

\section{Introduction}

Given a large odd number $n$ without small factors, one can try to decide whether $n$ is prime by randomly taking some $a$ coprime with $n$ and computing $a^{n-1} \mod n$. If this value is not $1$, then $n$ is certainly not prime, by Fermat's little theorem. Otherwise we can only say that $n$ is probably prime. Actually either $n$ is prime or $n$ is \emph{pseudoprime} for the base $a$; the latter is equivalent to saying that $a$ is a liar to Fermat's primality test.

Even if Fermat's primality test is often correct, unfortunately it cannot be trustingly used as a Monte-Carlo primality test because there exist odd composite numbers that are pseudoprimes for all of the bases coprime with $n$. These numbers are called \emph{Carmichael numbers}: they are much rarer than primes but they are still infinite, as proved by Alford, Granville and Pomerance in \cite{AGP}.

Instead of considering Fermat's little theorem one could use Euler's criterion: namely Euler proved that if $p$ is an odd prime, then $a^{\frac{p-1}{2}} \equiv ( \frac{a}{p} ) \mod p$ for every $a$, where $(\frac{a}{p} )$ is the Legendre--Jacobi symbol. Thus the primality of a large odd number $n$ can be tested by checking $a^{\frac{n-1}{2}} \equiv (\frac{a}{n})  \mod n$ for some $a$ coprime with $n$. If this relation is not satisfied then $n$ is certainly not prime; otherwise $n$ is probably prime and, as before, we have that either $n$ is prime or $n$ is an \emph{Euler pseudoprime} for the base $a$; the latter is equivalent to saying that $a$ is a liar to the Solovay--Strassen primality test.

In order to confidently use this primality test in a Monte-Carlo method, it was very important to establish for how many bases an odd composite number can be an Euler pseudoprime. It has been reported to the author by Pomerance that Selfridge was probably the first one to realize that for every odd composite number $n$ there is at least one $x$ for which $n$ is not an Euler pseudoprime for the base $x$, but he did not publish his discovery (see \cite[\S 5]{EP}, where Selfridge is credited). Anyway, a few years after Selfridge's discovery, both Lehmer (see \cite{L}) and Solovay--Strassen (see \cite{SS}), independently, proved the same result. In particular, Solovay and Strassen also noticed that the subset of bases in $U(\mathbb{Z}_n):=\{a \in \mathbb{Z}_n \mid \mathrm{gcd}(a,n)=1\}$ for which $n$ is an Euler pseudoprime is actually a subgroup. As an easy consequence of this fact, they showed that no odd composite number can be an Euler pseudoprime for more than half of the admissible bases (that is, elements of the group $U(\mathbb{Z}_n)$): 
\begin{equation} \label{disbase}
\left | \left\{a \in U(\mathbb{Z}_n) \mid a^{\frac{n-1}{2}} \equiv \left( \frac{a}{n} \right) \mod n\right\} \right |\leq \phi(n)/2.
\end{equation} 
This remark paved the way for an efficient probabilistic primality test, named Solovay--Strassen after the two authors.

During the subsequent years many papers about pseudoprimes, Euler pseudoprimes, strong pseudoprimes appeared: for example, an article by Pomerance, Selfridge and Wagstaff (\cite{PSW}), where many properties are stated and many examples are given and an article by Monier (\cite{M}), where a formula to count the number of liars is given. 

The purpose of this note is just to understand in which cases the bound in (\ref{disbase}) is actually achieved.

We will prove the following 
in an equivalent form as Proposition \ref{caratterizzazione}:
\begin{proposition}
Let $n$ be an odd composite number. Then $n$ is an Euler pseudoprime for exactly one half of the bases in $U(\mathbb{Z}_n)$ if and only if $a^{\frac{n-1}{2}} \equiv 1 \mod n$ for every $a \in U(\mathbb{Z}_n)$.
\end{proposition}


\subsection*{Acknowldegments}
The author wishes to warmly thank Prof.~ C.~ Pomerance for his kindness and for many helpful discussions.





\section{Preliminary lemmas}

\begin{lemma} \label{nomenouno} Let $n>2$ be an odd number. Then, for any $a \in  \mathbb{Z}$, $a^{n-1} \not \equiv -1 \mod n$.
\end{lemma}
\begin{proof}
Let $n=p_1^{\alpha_1}  \cdots  p_r^{\alpha_r}$, $p_i$ distinct prime numbers. Without loss of generality we can suppose that $\mathrm{gcd}(a,n)=1$
and that $v:=v_2(p_1-1) \leq v_2(p_i-1)$ for all $1 \leq i \leq r$ ($v_2$ being the dyadic valuation).
By contradiction, suppose $a^{n-1} \equiv -1 \mod n$. Then, in particular, $a^{n-1} \equiv -1 \mod p_1^{\alpha_1}$. Let $g$ be a generator for $U(\mathbb{Z}_{p_1^{\alpha_1}})$ 
and let $h$ be such that $g^h \equiv a \mod p_1^{\alpha_1}$. Then $g^{h(n-1)} \equiv -1 \mod p_1^{\alpha_1}$, that is, $$\frac{\phi(p_1^{\alpha_1})}{2} \mid h(n-1) \phantom{aaa}\textrm{    but    } \phantom{aaa} \phi(p_1^{\alpha_1}) \nmid h(n-1).$$ 
Hence there exists $k \in \mathbb{Z}$, $k$ odd, such that $\phi(p_1^{\alpha_1}) k = 2h (n-1)$, that is, $p_1^{\alpha_1-1}(p_1-1) k = 2h(n-1)$. It follows that $v=v_2(p_1-1) > v_2(n-1)$; however this is not possible since $p_i \equiv 1 \mod 2^v$ for every $ 1\leq i \leq r$ and thus $n-1 \equiv 0 \mod 2^v$.
\end{proof}

\begin{lemma} \label{carmichael} Let $n$ be an odd composite number and let $B:=\{a \in U(\mathbb{Z}_n) \mid a^{\frac{n-1}{2}} \equiv \pm 1 \mod n\}$. If $|B| \geq \phi(n)/2$, then $n$ is a Carmichael number.
\end{lemma}

\begin{proof}
If $|B|=\phi(n)$ the statement is trivial, therefore from now on we can suppose that $|B|=\phi(n)/{2}$. Let $B':= \{a \in U(\mathbb{Z}_n) \mid a^{\frac{n-1}{2}} \equiv  +1 \mod n\}$. It is easily seen that two cases can occur: either $B' = B$, that is, $B'$ has index $2$  in $U(\mathbb{Z}_n)$, or  $C:=B \setminus B' \not = \emptyset$ is a coset of $B'$ in $U(\mathbb{Z}_n)$, that is, $B'$ has index $4$. In the first case, for any $h \not \in B'$ we have $h^2 \in B'$, that is, $h^{n-1}={h^2}^\frac{n-1}{2} \equiv 1 \mod n$.  
In the second case we can conclude by observing that again $U(\mathbb{Z}_n)/B' $ has elements of order at most $2$. Indeed $C$ has order two in $U(\mathbb{Z}_n)/B' $, therefore if $h \not \in B$, then $h^2$ must be in $B'$ or in $C$, but the latter is not possible because otherwise $h^{n-1} = {h^2}^{\frac{n-1}{2}} \equiv -1 \mod n$, contradicting Lemma \ref{nomenouno}. 
\end{proof}

\begin{remark} \label{troppi}
Notice that in the proof of Lemma \ref{carmichael} the second case ($C \not = \emptyset$, $B'$ has index $4$) does not actually occur. In fact, since we have proved that $n$ is a Carmichael number, we now know that $n=p_1 \cdots p_r$ with $r \geq 3$, $p_i$ distinct primes (see, for example, \cite[Proposition V.1.2, Proposition V.1.3]{K}). For every $ 1 \leq i \leq r$, let $g_i$ be a generator of $U(\mathbb{Z}_{p_i})$. Since $C \not = \emptyset$, there exists $b \in U(\mathbb{Z}_n)$ such that $b^{\frac{n-1}{2}} \equiv -1 \mod n$, that is, $b^{\frac{n-1}{2}} \not \equiv 1 \mod p_i$ for every $i$. In particular $g_i^{\frac{n-1}{2}} \not \equiv 1 \mod p_i $ for every $i$. For every $1 \leq i \leq r$, let $x_i$ be the solution $\mod n$ of the system $X \equiv g_i \mod p_i, X \equiv 1 \mod n/p_i$. Notice that for every $i$, $x_i^{\frac{n-1}{2}} \not \equiv \pm 1 \mod n$ and that for any $i \not = j$, $x_i^{\frac{n-1}{2}} \not \equiv x_j^{\frac{n-1}{2}} \mod n$. Therefore, since $r \geq 3$, $B'$ must have index $\geq 5$.
\end{remark}

\begin{lemma} \label{jacobi}
Let $n>2$ be an odd number. Let  $$P_n:=\left\{a \in U(\mathbb{Z}_n) \mid \left( \frac{a}{n} \right) =1 \right\},$$ $$N_n:=\left\{a \in U(\mathbb{Z}_n) \mid \left( \frac{a}{n} \right) =-1 \right\}.$$
If $n$ is not a perfect square then $|P_n|=|N_n|=\phi(n)/2$.
\end{lemma}
\begin{proof}
Notice that we need only to prove that if $n$ is not a perfect square, then $N_n \not = \emptyset$, since in this case $N_n$ is just a coset of the subgroup $P_n$ in $U(\mathbb{Z}_n)$.
Let $n=p_1^{\alpha_1}  \cdots  p_r^{\alpha_r}$,  $p_i$ distinct primes. Without loss of generality we can suppose that  $\alpha_1$ is odd.
Choose $q \in U(\mathbb{Z}_{p_1})$ such that $q$ is not a quadratic residue $\mod p_1$. Let $x$ be any solution of 
\begin{displaymath}
\left\{ 
\begin{array}{l}
  X \equiv q \ \mod p_1\\
 X \equiv 1 \ \mod p_2  \cdots  p_r 
\end{array}
\right.
\end{displaymath}

Clearly $\text{gcd}(x,n)=1$. Moreover   $$\left( \frac{x}{n} \right) = \left( \frac{x}{p_1} \right)^{\alpha_1}  \cdots  \left( \frac{x}{p_r} \right)^{\alpha_r}=(-1)^{\alpha_1} =-1,$$ that is, $N_n \not = \emptyset$.
\end{proof}

\section{special Carmichael numbers}

\begin{definition}
Let $n$ be an odd composite number. We say that $n$ is a \emph{special Carmichael number} if $a^{\frac{n-1}{2}} \equiv 1 \mod n$ for all $a \in U(\mathbb{Z}_n)$.
\end{definition}

Following Korselt, we have the characterization below:
\begin{proposition} \label{korselt}
$n$ is a special Carmichael number if and only if $n$ is odd, square-free and $(p-1) \mid \frac{n-1}{2}$ for every prime $p$ such that $p \mid n$.
\end{proposition} 
\begin{proof}
Just a minor modification of Korselt's proof is needed (see, for example, \cite[Proposition V.I.2]{K}).
\end{proof}

\begin{remark}
It is clear that special Carmichael numbers are Carmichael numbers. As explained in \cite[Exercise 3.24]{CP} the proof of the infinitude of Carmichael numbers (see \cite{AGP}) actually implies that there are infinitely many special Carmichael numbers. The least example is the famous ``taxicab'' number $1729$, dear to Hardy and Ramanujan.  The first elements of the sequence of special Carmichael numbers are $1729, 2465, 15841, 41041,$ 46657$,$ 75361$,$ 162401,$ $ $172081, $ $399001, 449065, 488881, \ldots $. It is also clear from Proposition \ref{korselt} that every special Carmichael number must be $\equiv 1 \mod 4$. 
\end{remark}

We are now ready to prove our main result.

\begin{proposition} \label{caratterizzazione}
Let $n$ be an odd composite number. Then $n$ is an Euler pseudoprime for exactly one half of the bases in $U(\mathbb{Z}_n)$ if and only if $n$ is a special Carmichael number. 
\end{proposition}
\begin{proof}
If $n$ is a special Carmichael number then, in particular, $n$ is a Carmichael number and thus $n$ is square-free. Therefore, by Lemma \ref{jacobi}, $n$ is an Euler pseudoprime for half of the bases in $U(\mathbb{Z}_n)$, namely for all $a \in U(\mathbb{Z}_n)$ such that $\left( \frac{a}{n}\right) = 1$.

Conversely, suppose that $n$ is an Euler pseudoprime for exactly one half of the bases in $U(\mathbb{Z}_n)$. Then by hypothesis $a^\frac{n-1}{2} \equiv \pm 1 \mod n$ for at least one half of the admissible bases and therefore, in particular, $n$ is a Carmichael number by Lemma \ref{carmichael}, thus $n=p_1 \cdots p_r$, $p_i$ distinct primes, $r \geq 3$.
By \cite[Exercise 3.24]{CP} and Remark \ref{troppi}, either $a^{\frac{n-1}{2}} \equiv 1 \mod n$ for every $a \in U(\mathbb{Z}_n)$ or $a^\frac{n-1}{2} \equiv  1 \mod n$ for exactly one half of the admissible bases (while $a^{\frac{n-1}{2}} \not \equiv \pm 1 \mod n$ for the other half). We must rule out the latter case. 


By hypothesis and by Lemma \ref{jacobi}, 
$\left( \frac{a}{n}\right) = 1$ for all $a \in U(\mathbb{Z}_n)$ such that $a^{\frac{n-1}{2}} \equiv 1 \mod n$, while $\left( \frac{a}{n}\right) = -1$ for all $a \in U(\mathbb{Z}_n)$ such that $a^{\frac{n-1}{2}} \not \equiv 1 \mod n$. We will now exhibit $x \in U(\mathbb{Z}_n)$ such that $x^{\frac{n-1}{2}} \not \equiv 1$ but $\left( \frac{x}{n}\right)=1$, a contradiction.

Let $b \in U(\mathbb{Z}_n)$ such that $b^{\frac{n-1}{2}} \not \equiv 1$. In particular, there exists a prime factor $p$ of $n$, say $p_1$, such that $b^{\frac{n-1}{2}} \not \equiv 1 \mod p_1$. Let $g$ be a generator of $U(\mathbb{Z}_{p_1})$, so that $g^{\frac{n-1}{2}} \not \equiv 1 \mod p_1$. Let $g'$ be a generator of $U(\mathbb{Z}_{p_2})$. Let $x$ be the unique solution $\mod n$ of the system
\begin{displaymath}
\left\{ 
\begin{array}{l}
  X \equiv g \ \mod p_1\\
 X \equiv g' \ \mod p_2 \\
 X \equiv 1 \ \mod p_3   \cdots  p_r 
\end{array}
\right.
\end{displaymath}
We see immediately that $x^{\frac{n-1}{2}} \not \equiv 1 \mod n$ 
and $$\left( \frac{x}{n} \right) = \prod_{i=1}^r \left( \frac{x}{p_i} \right)=(-1)(-1)=1.$$

\begin{remark} \label{monier}
As Pomerance kindly pointed out to the author, Proposition \ref{caratterizzazione} can also be proved by a careful  consideration of all the cases in Monier's formula for the number of liars to the Solovay--Strassen test (see \cite[Proposition 3]{M} and \cite{EP}). Actually Monier, in \cite{M}, also observes that odd composite numbers achieving the bound in (\ref{disbase}) are Carmichael numbers, but he does not make calculations explicit and misses to give a complete characterization (although he was probably aware of the gist of Proposition \ref{caratterizzazione}).
It is also worth remarking that Monier, in \cite{M}, additionally gives a formula for the number of liars to the Miller--Rabin test. This formula can be used to give a complete characterization of odd composite numbers $n$ achieving the bound $\phi(n)/4$ for strong pseudoprimes: see \cite{M} and \cite[Equation 1.5 and \S 5]{ZT}.
\end{remark}

\end{proof}

\bibliographystyle{plain}  
\bibliography{pseudoprimes}

\end{document}